\documentclass[12pt]{amsart}
\usepackage{amscd}
\usepackage{verbatim}
\usepackage{amssymb, amsmath, amsthm, amscd,ifthen}
\usepackage[dvips]{graphics}
\usepackage[cp866]{inputenc}
\usepackage{graphicx}
\usepackage{epsfig}

\usepackage{amsmath,amssymb,amscd,}
\usepackage[all]{xy}

\usepackage[dvips]{graphics}
\usepackage[cp866]{inputenc}
\usepackage{graphicx}
\usepackage{epsfig}

\theoremstyle{plain}
\newtheorem{theorem}{Theorem}
\newtheorem{lemma}{Lemma}
\newtheorem{corollary}{Corollary}

\theoremstyle{definition}
\newtheorem{definition}{Definition}

\theoremstyle{plain}
\newtoks\thehProclaim
\newtheorem*{Proclaim}{\the\thehProclaim}

\theoremstyle{definition}
\newtoks{\thehRemark}
\newtheorem*{Remark}{\the\thehRemark}

\renewcommand{\leq}{\leqslant}
\renewcommand{\geq}{\geqslant}

\begin{document}

\title{Minimal triangulations of circle bundles}

\author{ Gaiane Panina, Maksim Turevskii}

\address{ G. Panina: St. Petersburg Department of Steklov Mathematical Institute; St. Petersburg State University;  gaiane-panina@rambler.ru; M.Turevskii:St. Petersburg State University; turmax20052005@gmail.com }

\subjclass[2000]{}


\begin{abstract}   A triangulation of a circle  bundle $ E \xrightarrow[\text{}]{\pi} B$ is a triangulation of the total space $E$ and the base $B$ such that the projection $\pi$ is a
simplicial map.

In the paper we address the  following questions:
 \textit{Which circle bundles can be triangulated over a given triangulation of the base?  What are the minimal triangulations of a bundle?}

 A complete solution  for  semisimplicial triangulations was given by N. Mn\"{e}v.
Our results deal with classical triangulations, that is, simplicial complexes. We give an exact  answer for an infinite family
of triangulated spheres (including the boundary of the $3$-simplex,
   the boundary of the   octahedron, the suspension over an $n$-gon, the icosahedron).  For the general case  we present a sufficient criterion
   for existence of a triangulation.
   Some minimality results follow straightforwadly.

\end{abstract}

\maketitle

\section{Introduction}\label{SecIntro}

Any smooth map of constant rank between two closed smooth
manifolds can be \textit{triangulated}, that is,   the manifolds can be triangulated  in such  a way that the map becomes simplicial.
Thus the difficult question of finding  minimal triangulations makes  sense not only for individual manifolds, but also for  fiber bundles.

Let $ E \xrightarrow[\text{}]{\pi} B$   be a \textit{circle bundle},  that is,  a principal $S^1$-bundle, or, equivalently, an oriented locally trivial  fiber bundle whose fibers are circles $S^1$. A \textit{ (classical) triangulation}  of the bundle is a representation of $B$ and $E$ as (supports of) finite simplicial complexes such  that $\pi$ is  a \textit{simplicial map}.

\medskip

\medskip
In the paper we address the  following question:

  \textit{Which circle bundles can be classically triangulated over a given triangulation of the base? What is the minimal triangulation (that is, a triangulation with the minimal number of vertices in $E$) of a given bundle?}

\medskip

\medskip
We will make use also of \textit{semi-simplicial triangulations}, known also as $\Delta$-\textit{sets}, see \cite{Hat}.  Semi-simplicial triangulations are much less restrictive than classical ones (for instance, they allow a simplex glue to itself), and therefore, they leave more freedom for triangulating.

\newpage
\subsection*{ What is known}
\begin{itemize}
  \item  A minimal (classical)  triangulation of the Hopf bundle over the boundary of the tetrahedron  is constructed in \cite{MS}. Its total space has $4\cdot 9$  three-dimensional simplices  and $12$ vertices.
      
      \item Only bundles with Euler number  $\mathcal{E}=0$ or $\pm 1$ can be classically triangulated over the boundary  of the tetrahedron  (see \cite{Mnev2} and \cite{Gan}).
  \item The local combinatorial formula (see Section \ref{SecPrel}  and \cite{Mnev3})) implies  that  if a bundle over a triangulated $2$-dimensional
orientable surfaces $B$ can be classically triangulated, then the base $B$ has more that $2|\mathcal{E}|$ triangles, where $\mathcal{E}$
is the Euler number of the bundle (see \cite{Mnev2} and  a short reminder below.).
  \item For semi-simplicial triangulations
N. Mn\"{e}v proved:

\medskip

\begin{theorem}\cite{Mnev1}\label{ThmMnev}
  \textit{ A circle bundle can be semi-simplicially triangulated over a
finite semi-simplicial base  $B$ iff its integer Euler class
can be represented by a simplicial cocycle  having values $0$ and $1 $ on $2$-simplices of the base. For classical simplicial triangulations the
condition is necessary but not sufficient.
}
\end{theorem}

\end{itemize}

\medskip

\subsection*{ Main results of the paper}

We start with $2$-dimensional
orientable surfaces as the base. In this case oriented circle bundles are classified by their integer Euler classes, or \textit{Euler numbers}, that is, elements of
$H^2(B,\mathbb{Z})=\mathbb{Z}$.

\begin{theorem}\label{ThmMain1}  Let  $B$ be a triangulated two-dimensional closed oriented surface with $f(B)$  triangles  and $v(B)$ vertices.

                                 If   $f(B)\geq 4|\mathcal{E}|,$
 then the circle bundle with  Euler number $\mathcal{E}\in \mathbb{Z}$  can be classically  triangulated over the given triangulation of
$B$ with  $3v(B)$ vertices and $9 f(B)$  three-dimensional simplices in the total space $E$. We present an explicit construction of such a triangulation. 

                                  A smaller number of  $3$-dimensional simplices in  $E$  is not possible for the given triangulation of $B$.

\end{theorem}

For very many simplicial $2$-spheres, the condition  $f(B)\geq 4 |\mathcal{E}|$ in Theorem \ref{ThmMain1} is not only sufficient, but also necessary.
For a precise statement we need the following:


\textbf{Coloring game}

The input of the game is an (uncolored) simplicial  $2$-sphere $B$.
\begin{enumerate}
  \item Pick  either one or two vertices of $B$, and color them  red.
  \item  Pick an edge with two currently not colored vertices that forms a face with an already colored vertex. Color the  edge and its  vertices red,
and color the  face green, see Fig. \ref{FigColor}.

Repeat (2) until it is possible.
\end{enumerate}

We say that the game has \textit{a winning strategy} if eventually  the number of green faces is greater or equal than $f(B)/4$. 

\begin{figure}[h]\label{FigColor}
\begin{center}
\includegraphics[width=10cm]{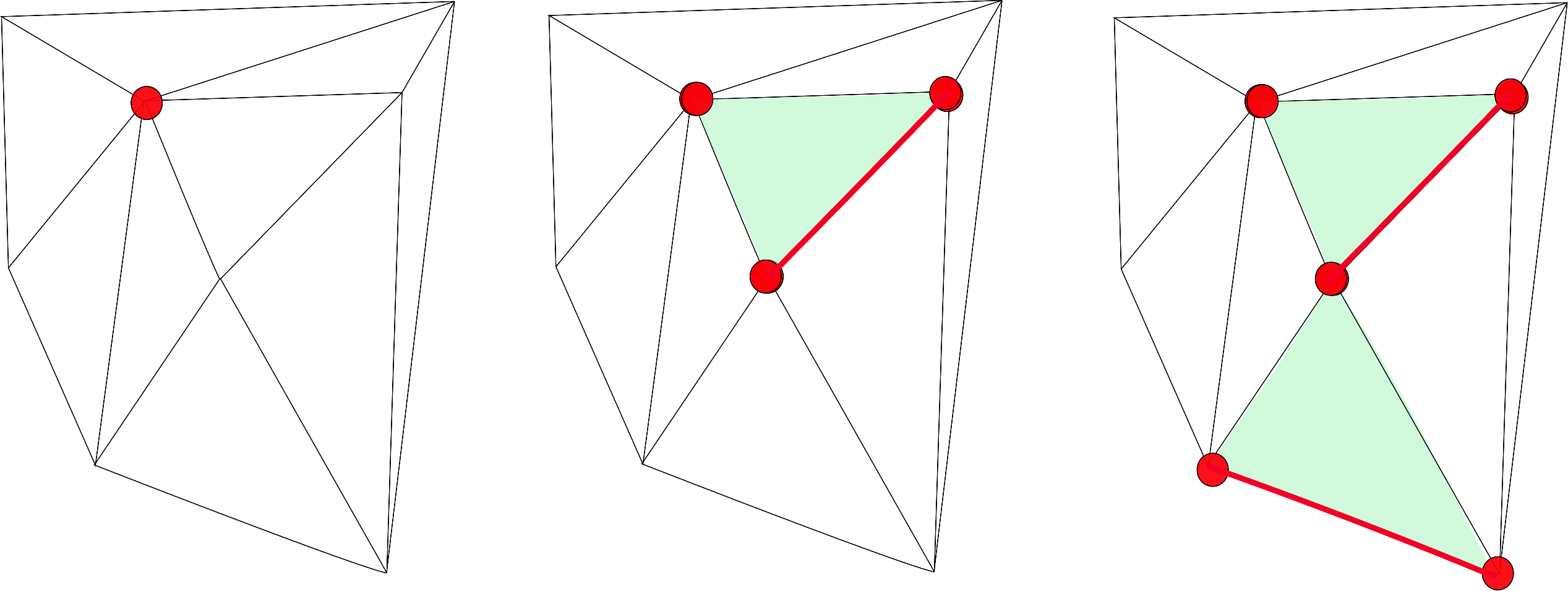}
\caption{Two steps of a coloring game}
\end{center}
\end{figure}

\begin{theorem}\label{ThmMain2}

  Assume a triangulated  sphere $B$  has a winning strategy for the coloring game. Then
  the circle bundle with  Euler number $\mathcal{E}\in \mathbb{Z}$  can be classicaly  triangulated over $B$   iff  $$f(B)\geq 4|\mathcal{E}|.$$

\end{theorem}

\medskip

	\begin{lemma}\label{LemWin}
		All triangulated spheres with at most $12$ triangles have a winning strategy.
	\end{lemma}
	
	\begin{corollary}\label{CorMin}
		A minimal triangulation of a circle bundle with Euler number $|\mathcal{E}|\leq 3$  over the two-dimensional sphere  has  $6+6|\mathcal{E}|$ vertices and $36|\mathcal{E}|$ three-dimensional simplices in the total space $E$.  It can be constructed by choosing an arbitrary  triangulation of the sphere with $4|\mathcal{E}|$ triangles, and then following the instruction of Section \ref{SecMain1}.
	\end{corollary}
	\medskip
	
	\medskip
	\textbf{Non-example}  However, there exist triangulated spheres with no winning strategy:
 take a bipyramid, that is, a triangulation of the sphere with $6$ faces. Next, take the stellar subdivision of all the $2$-faces. That is,
	we add a vertex in the center of each of the $6$ triangles, and connect each new vertex to the vertices of the triangle it belongs to.
Altogether  we have $5$ \textit{old}  vertices (initial vertices of the bipyramid), and the $6$ \textit{new} vertices. 
	
	 Each edge necessarily contains an old
	vertex, and by an easy case analysis one concludes that the number of red edges  is at most $4$, which does not suffice, since $f(B)=18$.
	
	\medskip

Now we turn to the general case, the base $B$ is no longer a surface.

\begin{theorem}\label{ThmMain3}

  The trivial circle bundle over any triangulated  base $B$ can be triangulated with $3v(B)$ vertices in the triangulation
  of the total space $E$.

\end{theorem}

To generalize Theorem \ref{ThmMain1} in the spirit of Theorem \ref{ThmMnev}  for arbitrary finite  simplicial complex  $B$,  we need some preparation.
Define a function  $$\mathcal{F}:\{\pm 1,\pm 3\}\rightarrow \{\pm 1/4\}$$  by $$\mathcal{F}(3)=\mathcal{F}(-1)=1/4,  \hbox{ and } \mathcal{F}(-3)=\mathcal{F}(1)=-1/4.$$

Consider a simplicial $1$-cochain  $a\in C^1(B, \mathbb{Z})$ attaining the values $\pm 1$ only. Its coboundary $da$ (here and in the sequel, $d$ is the coboundary operator) is a $2$-cochain attaining the values
from the set $\{\pm 1,\pm 3\}$ only.  Set    $$\mathcal{ G}_a=\mathcal{F}(da)-\frac{3\cdot da}{4} \in C^2(B,\mathbb{Z}).$$  One easily checks that $\mathcal{G}_a$ is an integer  $2$-cochain admiting the values $\pm 1$ and $\pm 2$ only.

\begin{theorem}\label{Thm}  Let $E\rightarrow B$ be a circle bundle over a simplicial complex $B$.
 Assume that  the group $H^2(B,\mathbb{Z})$  has no  elements of order two.
 
Assume there exists a simplicial cocycle $e$ attaining values $\pm1$ and $\pm 2$ on $2$-simplices of the base such that two conditions hold:

(1)  $e$ represents the integer Euler class
of the bundle.

(2) There exists    a simplicial $1$-chain  $a\in C^1(B, \mathbb{Z})$ admitting the values $\pm 1$ only   such that   $e=\mathcal{G}_a$.

Then the circle bundle can be classically  triangulated over the given triangulation of
 $B$.
\end{theorem}

\section{Preliminaries}\label{SecPrel}
	In this section we first present a convenient way to describe  a circle bundle by a collection of necklaces and morphisms. Next, in the language of necklaces, we present
	the local combinatorial formula for the rational Euler class. In the section we follow (not literally) the papers \cite{Mnev1} and \cite{Mnev3}.

	\subsection{Necklaces} \textit{A necklace} is a cyclic word with letters from an alphabet.

	Each triangulated circle  bundle $ E \xrightarrow[\text{}]{\pi} B$ yields a collection of necklaces, a necklace per each simplex of the base, and a collection of morphisms.
	
	Here how it goes, see Fig. \ref{FigNecklace1}:\begin{enumerate}
		\item \textbf{A necklace over a simplex $\sigma^k$ from the base.}
		
		The restriction of a circle  bundle to a  $k$-dimensional simplex $\sigma^k\in B$  is a triangulation of $\sigma^k\times S^1$.
		Its combinatorics    is described in terms of a {necklace} $\mathcal{O}(\sigma^k)$, whose beads are labeled  (= \textit{colored}) by the vertices of the simplex  $\sigma^k$.

		Let a simplex $\rho^{k+1} \in \pi^{-1}(\sigma^k)$. 
		There is exactly one vertex in $Vert(\sigma)$ which has an edge of $\rho^{k+1}$  in the preimage. Let us label
		$\rho^{k+1}$ by this vertex.  Now take an inner point $x$ of $\sigma^k$. Its preimage $\pi^{-1} (x)$ is a circle
		intersecting all the $k+1$-dimensional simplices in $\pi^{-1}(\sigma^k)$. Their labels form  the necklace $\mathcal{O}(\sigma^k)$.
		
		\item \textbf{Morphisms.}
		If    $\sigma^i\subset \sigma^k$  then the necklace $\mathcal{O}(\sigma^i)$ is obtained from $\mathcal{O}(\sigma^k)$ by deleting of all the beads that are labeled by the vertices from $ \sigma^k\setminus \sigma^i$. In other words, there is an (injective) morphism $$\phi_{ij}:\mathcal{O}(\sigma^i)\rightarrow \mathcal{O}(\sigma^j),$$
		which maps the beads keeping their labels and the cyclic order.
		
		Morphisms are \textit{consistent}, that is, satisfy the natural properties: whenever $ \sigma^i\subset \sigma^j\subset \sigma^k$, we have $\phi_{ik}=\phi_{jk}\circ \phi_{ij}$, and $\phi_{ii}=id$.
		
	\end{enumerate}
	
An example of a collection of  necklaces associated  with two adjacent triangles and their common edge is presented in Fig. \ref{FigNecklaces}.

\begin{figure}[h]
\includegraphics[width=6cm]{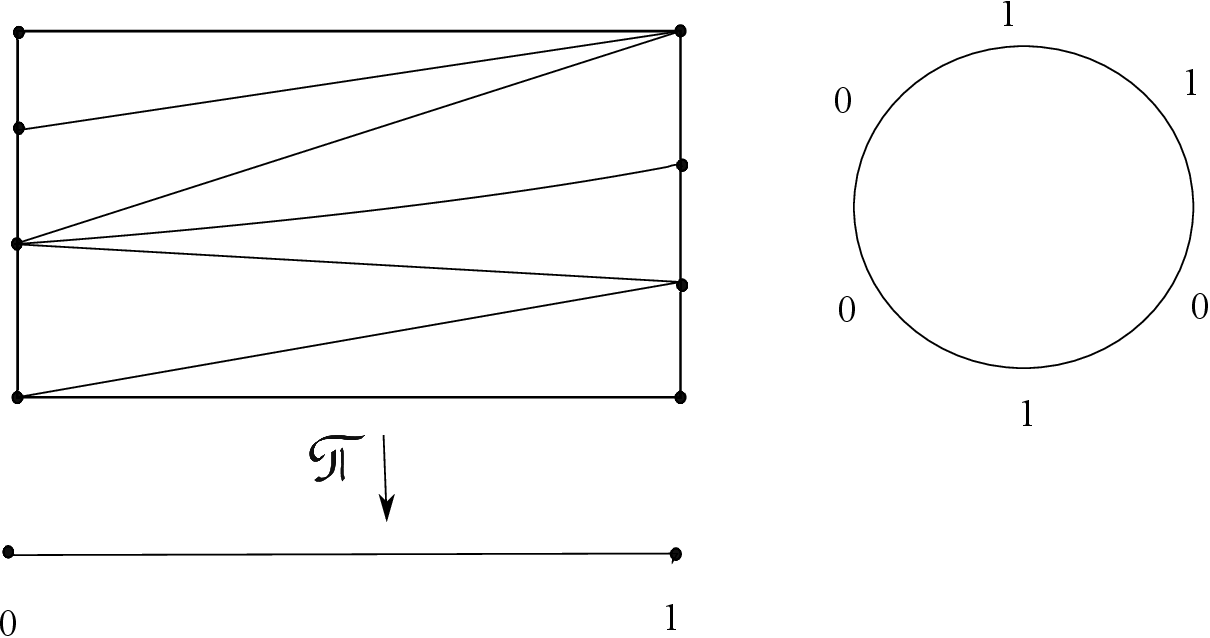}
\caption{A necklace over a one-dimensional simplex. We depict the  cylinder $\pi^{-1}([01])$ as a triangulated quadrilateral, assuming that the upper and the bottom sides are glued. }
\label{FigNecklace1}
\end{figure}

\begin{figure}[h]
\includegraphics[width=10cm]{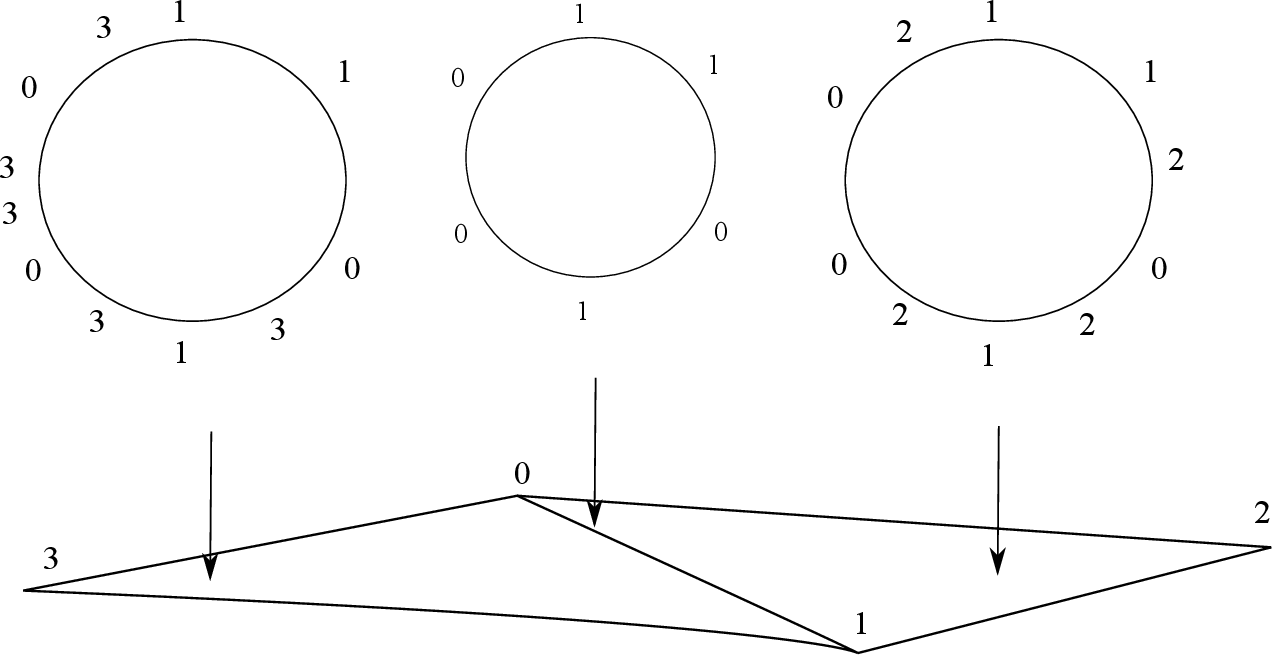}
\caption{ This  picture  describes a triangulated fiber bundle over the union of the two triangles.}
\label{FigNecklaces}
\end{figure}

	\medskip
	We record necklaces as non-cyclic words, assuming that a local numbering $0,1,..,k$ for the vertices  of the simplex $\sigma^k$   is fixed. We write e.g., $O(\sigma^3)=021013$, keeping in mind that  $021013=210130=101302=013021.$
	\medskip
	
	\medskip
	It is proven in \cite{Mnev1} (also easy to see) that given a triangulated base $B$, each consistent collection of necklaces and morphisms defines a semisimplicially triangulated  circle bundle.
	Indeed, an individual necklace restores the combinatorics of the preimages of a simplices, and the morphisms give a way of patching together all these trivial pieces.

	\newpage
	
	There is a simple condition of   classical simplicial triangulations:
	
	\begin{lemma}\label{LemmaMixed}\cite{Mnev1}
		A consistent collection of necklaces and morphisms yields a classicaly simplicialy  triangulated circle bundle iff
		\begin{enumerate}
			
			\item For each of the necklaces there are no less than $3$ beads of each color, and
			
			\item for  all the  two-colored necklaces, the colors are \textit{mixed}, that is, the beads do not come in two monochromatic blocks.
			
			E.g., in the necklace $010111$  the colors are  mixed, whereas in $000111$  they are not.
		\end{enumerate}
	\end{lemma}

	\medskip

	\subsection{Local combinatorial formula (LCF)}

	The local combinatorial formula for the Euler class of a triangulated oriented circle bundle appeared independently in \cite{Mnev3} and \cite{Igusa}.
	It reads as follows.

	Given a triangulated  circle bundle, define the following  $2$-cochain $\varepsilon$:
	
	Pick $\sigma^2 \in B$. We may assume that its vertices are numbered by $0,1,2$, and the numbering agrees with the orientation of the simplex.
	According to our construction, the circle bundle yields a necklace in these three letters.  Each \textit{multicolored triple of beads}, that is a  triple of beads colored by $0$, $1$, and $2$, are positioned on the necklace either with\textit{ positive} or with \textit{negative orientation},  that is, either $012$ or $021$ respectively.
	By definition, the cochain $\varepsilon$ assigns to $\sigma^2 $ the rational  number  $$\varepsilon(\sigma^2)=\frac{\sharp(neg)-\sharp(pos)}{2\cdot \sharp(0)\cdot \sharp(1)\cdot \sharp(2)}.$$
	
	Here  $\sharp(neg)$ is the number of negatively oriented multicolored triples in the necklace $\mathcal{O}(\sigma^2)$,
	$\sharp(pos)$ is the number of positively oriented multicolored triples, $\sharp(i)$ is the number of beads of color $i$.
	\begin{theorem}
		The cochain $\varepsilon$ is a cocycle representing the rational Euler class of the bundle.
		
	\end{theorem}
	
	\textbf{Remark:}
	Since $\sharp(0)\cdot \sharp(1)\cdot \sharp(2)=
	\sharp(neg)+\sharp(pos)$  is the number of all multicolored triples,
	the  absolute value of the cochain  $\varepsilon$ on a simplex does not exceed $1/2$ (for semisimplicial triangulations)  and is strictly smaller than $1/2$ (for classical simplicial triangulations).
	
\medskip

	If the base $B$ is an oriented closed $2$-dimensional surface, the Euler number equals the sum of the values of the cochain over all the triangles in $B$:
	$$\mathcal{E}=\sum_{\sigma^2 \in B}\varepsilon(\sigma^2).$$
	Since each of the summands does not exceed $1/2$, we conclude that $f(B)\geq 2\mathcal{E}$.

	\section{Auxiliary construction: small  necklaces and small semisimplicial triangulations. Proof of Theorem \ref{ThmMain3}}
	Let $B$ be an orientable triangulated closed surface. We assume that all its $2$-simplices are oriented consistently with the global orientation of the surface.
	We shall build a consistent collection of necklaces with   morphisms.
	
	\subsection{Small  necklaces}
	
	A necklace which is centrally symmetric and contains exactly two beads of each color is called a\textit{ small symmetric necklace} or a \textit{small necklace}  for short. Such necklaces will be our building blocks.
	
	\medskip
	
	\textbf{Example:} \begin{enumerate}
		\item There exists only one  type (up to an isomorphism) of two-colored small  necklaces, that is, $0101$.
		\item There exist two types of three-colored small necklaces,
		$012012$ and $021021$.  According to the LCF, they contribue  the values of the cochain   $-1/4$ and $1/4$ respectively.
		
	\end{enumerate}

	\medskip
	
	\textbf{Example:}
	There are two morphisms from  $0101$  to $012012$. To record one of them, one chooses  a bead and makes it bold.
	We assume that a morphism maps a bold bead to a bold one.
	So the two  (uniquely defined by bold beads) morphisms are:
	\begin{enumerate}
		\item  $\mathbf{0}101$  to $012\mathbf{0}12$, and
		\item  $\mathbf{0}101$  to $\mathbf{0}12{0}12$.
	\end{enumerate}
	
	\medskip
	
	This example motivates the following
	\begin{definition}
		A \textit{framed  small necklace}  is a small necklace  with one bold bead for each color.
		\medskip

		A (semisimplicial) triangulation  of a circle bundle  is called \textit{small }  if all the associated necklaces are small ones.
	\end{definition}

	To build a  small  triangulated   circle bundle,  we  create a consistent collection of framed small necklace, one necklace per a simplex of the base.  \textit{Consistency} means that  when passing to a face  $\tau$ of a simplex $\sigma$,  the bold  beads are preserved.
	
	\medskip
	
	Take a triangle in the base  and label its vertices by the numbers $0,1,2$ such that the orientation agrees with the orientation of $B$.
	Let us list all possible small framed necklaces for the triangle and compute their contribution to LCF.

\begin{lemma}\label{LemmaComput}
  Small necklaces contributing   $-1/4$ to LCF are:  $\mathbf{012}012$,  $\mathbf{01}201\mathbf{2}$, $\mathbf{0}120\mathbf{12}$, $\mathbf{0}1\mathbf{2}0\mathbf{1}2$,
	
	Small necklaces  contributing   $1/4$ to LCF are: $\mathbf{021}021$,  $\mathbf{02}102\mathbf{1}$, $\mathbf{0}210\mathbf{21}$, $\mathbf{0}2\mathbf{1}0\mathbf{2}1$.
\end{lemma}

	\medskip
	
	A framing defines an orientation on the edges of the triangle:  the edge $ij$   of the triangle  is oriented as  $\overrightarrow{ij}$ iff moving along the necklace  $\mathcal{O}(ij)$  from a bold bead $i$ in the positive  (counterclockwise) direction,  one first meets the bold bead $j$.
	
	Examples are given in Fig. \ref{FigFr}.

\begin{figure}[h]
\includegraphics[width=8cm]{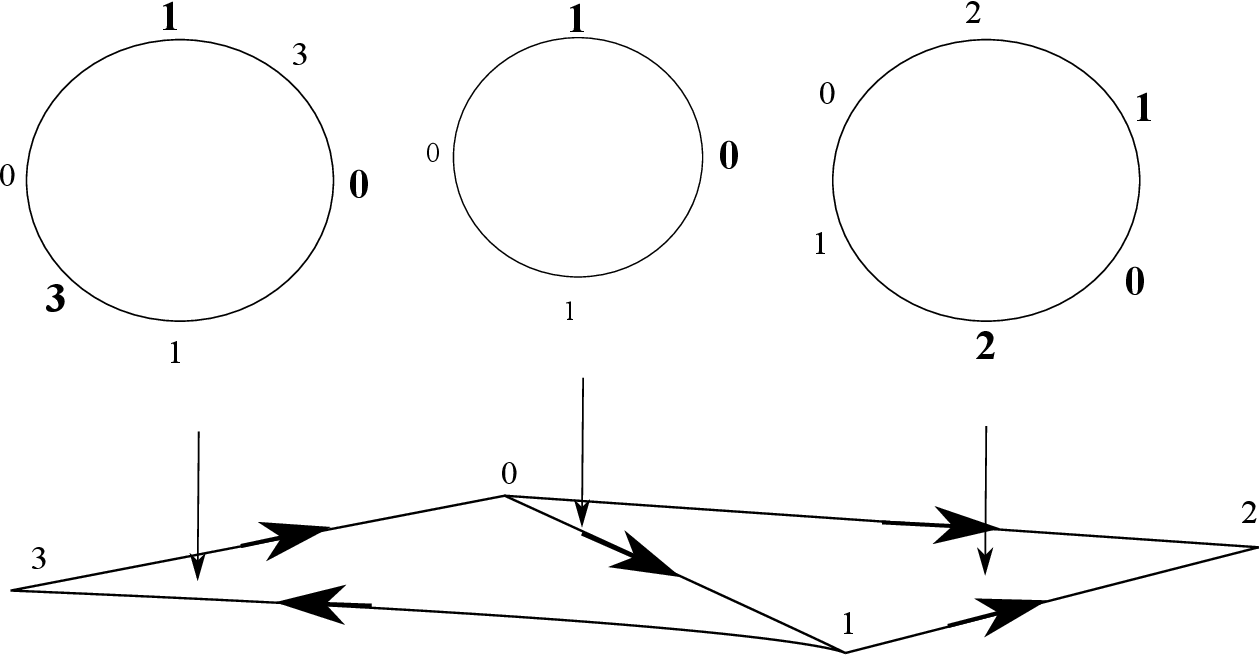}
\caption{A framing defines an orientation on the edges. }
\label{FigFr}
\end{figure}

	\medskip
	
	For a small triangulation of a bundle  over $B$
	we say that an edge of a triangle $\sigma^2\in B$  is \textit{positive  in $\sigma^2$ }  if its orientation agrees with the orientation of the triangle (which in turn agrees with the global orientation of the base), and \textit{negative in $\sigma^2$ }  otherwise.

	One immediately checks that a framed small necklace contributes $1/4$ iff the number of its positive edges is odd.
	
	\medskip
	
	Denote by $Or(B)$ the set of all possible orientations of the edges of $B$. Alternatively $Or(B)$ can be viewed as the set of all $1$-cochains admitting values $\pm 1$ only.
	\begin{lemma}\label{lemmaSmall} \begin{enumerate}
			\item Small  framed   (semisimplicial) triangulations  of  circle bundles with a given triangulated orientable two-dimensional base $B$
			are in a natural bijection with $Or(B)$.
			\item For $a\in Or(B)$, the  $2$-cochain $\mathcal{F}(da)$  is a representative of the Euler class of the circle bundle associated with  the orientation $a$.  Here $d$ is the coboundary operator, $\mathcal{F}$ is defined in Section \ref{SecIntro}.
		\end{enumerate}
		
	\end{lemma}
	\begin{proof}
		(1) We  have defined above  a way to associate an element of $Or(B)$ with  a small  framed   triangulation.
		
		Conversely, if an orientation is fixed, assign to each of the triangles the corresponding small  framed symmetric necklace.
		The necklaces are consistent, and therefore, taken together, they define a semisimplicial triangulation of some bundle.
		
		(2) This follows from Lemma  \ref{LemmaComput}  and LCF.
	\end{proof}

	\medskip
	
	Now we can prove Theorem \ref{ThmMain3}.

Enumerate the vertices of $B$ in an arbitrary way by $0,...,N$. Thus $B$ can be viewed as a subcomplex of the simplex $\Delta^N$.
Associate the necklace $$\mathcal{N}=(0,0,1,1,2,2,3,3,...,N,N,\mathbf{0},\mathbf{1},\mathbf{2},\mathbf{3},...,\mathbf{N})$$  to the simplex $\Delta^N$. Associate to each face of $\Delta^N$ the
necklaces obtained from $\mathcal{N}$ by eliminating all the beads that do not correspond to the face. For an example, the necklace $(0,0,2,2,3,3,\mathbf{0},\mathbf{2},\mathbf{3})$ is associated with the triangle $023$. The triangulated fiber bundle  which comes from this collection of necklaces is trivial since the base is contractible. Therefore the restriction of the bundle to $B$ is also trivial.
	
	\medskip
	\section{Proof of Theorem  \ref{ThmMain1}}\label{SecMain1}

	\subsection*{A small  semisimplicial triangulation with a prescribed $\mathcal{E}$}
	
	Let  $B$ be a triangulated two-dimensional closed oriented surface
	with $f(B)\geq 4|\mathcal{E}|$.    Let us  construct a  small  semisimplicially triangulated circle bundle with Euler
	number $\mathcal{E}\in \mathbb{Z}$.
	
	Fix one of the simplices $\sigma^2_0\in B$.
	Assuming that each $2$-simplex of $B$ is oriented consistently with the global orientation of $B$, fix any $2$-cochain  $b$  which assigns $\pm 1/4$ to each of the $2$-simplices
	such that the sum of the values equals  $\mathcal{E}$.

	According to Lemma \ref{lemmaSmall}, we need to find $a\in Or(B)$  such that $b=\mathcal{F}(da)$.
	We arrive at a collection of  $f(B)-1$  linear equations over the field $\mathbb{Z}_2$:  we have one equation for each of the $2$-simplices except for $\sigma^2_0$, the unknowns are the orientations of the edges.

In simple words, assume that $b(\sigma^2_i)=1/4$. The corresponding equation rephrases as "the number of edges  of $\sigma^2_i$  whose orientations agree with
the orientation of  $\sigma^2_i$  is odd."

	Lefthand sides of the equations are linearly independent, so there exists a solution $a$  which gives the prescribed number $\pm1/4$ on each of the faces except, may be, for the face
	$\sigma^2_0$. 
By Lemma \ref{lemmaSmall},  $a$ yields some triangulated circle bundle. Applying the LCF, we get a cochain which (1) represents the Euler class of the bundle, and (2)
 is identical with $b$ except, may be, the unique simplex $\sigma_0$.    Since we deal with integer numbers, we have
 $b(\sigma^2_i)=\mathcal{F}(da)(\sigma^2_i)$ for all the simplices.

	\subsection*{Upgrade    to a classical simplicial triangulation}
	
	Each necklace of the semisimplicial triangulation satisfies the mixed colores property from Lemma \ref{LemmaMixed}, so it remains to
	increase the number of the beads  of each color by one.
	
	There is an operation of \textit{doubling of a bead} in the necklaces. One chooses a bead over a vertex of the base, and replaces the bead and all its images under morphisms
by a doubled bead. The doubling operation corresponds to a refinement of triangulation of the bundle, see Fig. \ref{E}.  This operation  is inverse to \textit{spindle contraction}  from \cite{Mnev1}.   This operation  doesn't change the Euler class.
	So one doubles one bead over each of the vertices of the base, and also consistently beads are added to all the necklaces. Eventually each of the necklaces over a $2$-simplex has $9$ beads, three of each color, which is the smallest possible for a classical simplicial triangulation.

\begin{figure}[h]
\includegraphics[width=14cm]{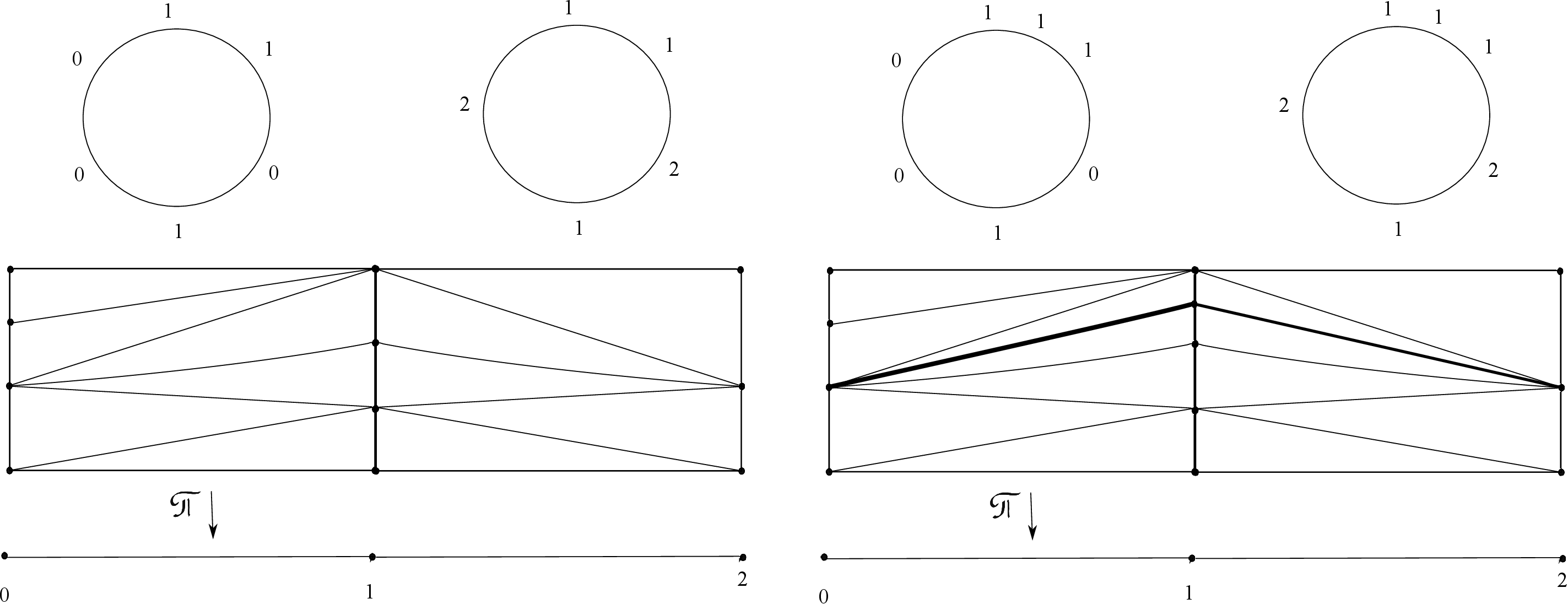}
\caption{Doubling operation. Left: initial triangulation and initial necklaces. Right: after doubling of a bead. }
\label{E}
\end{figure}

	\section{Proof of Lemma  \ref{LemWin}  and Corollary \ref{CorMin}}
	We have at most $f=12$ triangles in the triangulation  of $B$. Since  the number of vertices $v$ satisfies $v=(f+4)/2$,  we have at most $8$ vertices. Since for
the number of edges $e$ we have
	$e=3v-6$, we   always have a vertex with a degree at least $6-12/v$.

	Let $v=4$, that is, we have a tetrahedron. Let's pick two of its vertices, color them red,  and then pick the remaining two vertices and win.

	If $v=5$ or $v=6$,  there is a vertex $V$ with degree at least $4$. Color the vertex red,  and after that, pick two disjoint edges connecting the  neighbours of the vertex.
In the case $v=6$, there remains an uncolored vertex; In the case $v=6$, there remains an uncolored vertex; we may color it in the beginning..

	If $v=7$  or  $v=8$,  there is a vertex $V$ with degree at least $5$. Color $V$ red. If $V$ has degree at least $6$, there are three disjoint edges connecting  neighbours of $V$.
One colors them one by one and eventually wins. Otherwise $V$ has a neighbour  $V'$ which is connected with a  vertex $V''$   sharing no edge with $V$.  Let  us color red  two disjoint edges connecting the neighbors of $V$, but not containing $V'$. After that one colors the edge $V'V''$. (In the case $v=8$ there remains one uncolored vertex; we may color it in the beginning.)

	\medskip
	Now  Corollary \ref{CorMin}  is straightforward, since the number of $2$-simplices in the base cannot be smaller than $\mathcal{E}/4$, and the preimage of each of them
	contains at least nine $3$-simplices  by Lemma \ref{LemmaMixed}.

	\section{Proof of Theorem  \ref{ThmMain2}}
	
	\begin{proof} We need to prove the following: if the base $B$ has a winning strategy for the coloring game,  and $$f(B)< 4|\mathcal{E}|,$$
		then a circle bundle with  Euler number $\mathcal{E}\in \mathbb{Z}$  can not be (classicaly) simplicially triangulated over
		$B$.
		
		\medskip
		
		Assume the contrary, that is, for a triangulated circle bundle  $f(B)< 4|\mathcal{E}|$. Without loss of generity we may also  assume that $\mathcal{E}>0$.
		
		We make use of an operation of \textit{doubling} the bead in the necklace  (see Section \ref{SecMain1})  which  doesn't change the isomorphism class of the circle bundle.
		
		For a vertex in $B$ we  choose one bead and double it many (tending to $\infty$) times, so we get a \textit{huge bead}. We plan to create one huge bead above each of the vertices of the base. Application of the LCF for a face $\sigma^2$ with three   huge beads (one huge bead over each of the vertices of $\sigma^2$)   gives  the value $\varepsilon(\sigma^2)$ which is arbitrarily close to $\pm 1/2$. Let us say that the value is \textit{almost} $\pm 1/2$. We shall show that it is possible to create huge beads  such    that  the value of the  Euler number  of the bundle will be at most $f(B)/4$.
		
		We follow our coloring game:
		
		Firstly, we  pick one or two vertices, and select some huge beads for them. Next, we  take the adjacent green face. It has one more vertex which becomes colored on this step. Chose a bead over the new vertex of the green face such that the value of the LCF   for the new green face  gives a value  $\varepsilon(\sigma^2)$  which  {almost} equals $-1/2$. 
 Proceed in the same way
		adding a vertex and a green face at a time until all the vertices of the base are colored. According to construction, the value of the cochain on each of the green faces is    $-1/2$. On the other faces the value cannot exceed $1/2$. The number of green faces is at least $k=\lfloor (v-1)/2 \rfloor$. Therefore the Euler number is at most $(f-k)/2-k/2=f/2-k$. The Euler formula $v-e+f=2$ together with $2e=3f$ imply $v=f/2+2$, so  Euler class is at most $f/2-\lfloor ((f/2+2-1)/2) \rfloor=\lfloor f/4 \rfloor$.
	\end{proof}

	\section{Proof of Theorem  \ref{Thm}}
	
	The main ingredient of the proof is the fact that the (isomorphism class of) a circle bundle is uniquely determined by its integer Euler class.
So it suffices to construct a triangulated circle bundle with a prescribed Euler class. This will be done in three steps:
\begin{enumerate}
  \item We first construct a triangulated circle bundle over the $2$-skeleton of the base. Here we again use small symmetric necklaces  combined with doubling of some beads.
  \item Next, we show that this collection of necklaces extends to a collection of necklaces over the $3$-skeleton.  This is possible by cocycle property  of $e$.
  \item Finally, we show that this collection of necklaces (unconditionally) extends to a collection of necklaces over the entire base $B$.
\end{enumerate}

Cochains $a\in C^1(B, \mathbb{Z})$ admitting the values $\pm 1$ only are in a natural bijection with orientations on the edges of $B$.
 The value of $\mathcal{F}(da)(\sigma^2)$ at a simplex $\sigma^2$ equals $1/4$ iff the number of edges of the simplex  $\sigma^2$ whose orientation
agrees with the orientation of  $\sigma^2$ is even. Otherwise,  $\mathcal{F}(da)(\sigma^2)=-1/4$

	Assume the integer Euler class  of a circle bundle
	is represented by a simplicial $2$-cocycle $e=\mathcal{G}_a$  for some $a\in C^1(B, \mathbb{Z})$ admitting the values $\pm 1$ only. Then the  cocycle $\mathcal{F}(da)=\mathcal{ G}_a+\frac{3\cdot da}{4}$ also represents the same cohomology class   (since the difference of the two cocycles  is a coboundary).

	Create a small symmetric circle bundle over the $2$-skeleton of the base $B$ similarly as it was done in  Section \ref{SecMain1}.
		This gives a triangulated bundle over the $2$-skeleton of $B$ such that the LCF
	gives exactly the cochain $\mathcal{ F}(da)$.
	\begin{lemma}
		Since $\mathcal{ F}(da)$  is a cocycle, the  small symmetric triangulated circle bundle  (over the $2$-skeleton of $B$) extends uniquely to a small symmetric triangulated circle bundle over the $3$-skeleton.
	\end{lemma}

\begin{proof}

	The lemma is proven by a small case analysis. Take a  simplex $\sigma^3$ in the base. It is easy to see that there exist (up to symmetries) only two scenarios for the cochain $a$. One scenario  gives a $2$-cochain  attaining the values $1/4$ on each of the faces.
	This cochain is not a cocycle, and induces the Hopf bundle over $\partial \sigma^3$. The other scenario gives a cochain which takes the values $1/4$ on two faces of $\sigma^3$, and the values $-1/4$ on the other two faces.  This amounts to necklaces $\mathbf{123}123$,  $\mathbf{124}124$,  $\mathbf{234}234$, and  $\mathbf{134}134$, which can be incorporated to a necklace  $\mathbf{1234}1234$.
  
\end{proof}

	\begin{lemma}
		Any small symmetric triangulated circle bundle over the $3$-skeleton of $B$ extends uniquely to a small symmetric triangulation over $B$.
	\end{lemma}
\begin{proof}

In a slight disguise,  a version of the lemma is proven in \cite{Mnev1},  Prop. 7.
However, we give here a proof  for completeness.

	 Let us rephrase the lemma:  for $k \geq 5$, any collection of necklaces together with morphisms associated to the faces  of a $k-1$-simplex originates from some necklace with $k$ types of beads.

 Assume  that $k=5$, and let the vertices of the simplex be $0,1,...,4$.
	Fix  one bead   colored by $0$  in the necklace over the vertex $0$, and mark all its images in the other necklaces associated to faces of the simplex containing the vertex $0$. Choice of this bead defines a linear order on the beads in these necklaces (we treat this bead as the "first" one).  Indeed,  every two beads  $a$, $b$  appear in a necklace with colors $a,b$ and $0$, and thus define either $a<b$ or $b<a$.   If  $a < b$ and $b < c$, then we can look at the necklace with the colors of beads $a,b,c$ and $0$  and conclude that $a<c$.
So  we can assemble all the necklaces containing $0$ in a unique  a $k$-colored necklace. Each triple of beads has one and the same orientation which does not depend on the ambient necklace. Using this, it is easy to see that  the $k$-colored necklace     agrees with the  necklaces  from the collection that do not contain $0$.
  
\end{proof}


\end{document}